\documentclass[12pt]{article}      
\usepackage{amsmath,amssymb,amsthm, amscd, graphicx, verbatim, setspace} 
\usepackage{setspace}
\usepackage{mathtools}
\usepackage{tabstackengine}

\theoremstyle{plain}
\newtheorem{thm}{Theorem}

\newtheorem{cor}[thm]{Corollary}

\newtheorem{prop}[thm]{Proposition}
\newcommand\floor[1]{\lfloor#1\rfloor}
\newcommand\ceil[1]{\lceil#1\rceil}
\usepackage{algorithm}
\usepackage{algpseudocode}
\usepackage{amsthm}
\usepackage{amsmath}

\title{Forbidden arithmetic progressions in permutations of subsets of the integers}
\date{}
\author{Jesse Geneson}
\begin{document}
\maketitle

\begin{abstract}
Permutations of the positive integers avoiding arithmetic progressions of length $5$ were constructed in (Davis et al, 1977), implying the existence of permutations of the integers avoiding arithmetic progressions of length $7$. We construct a permutation of the integers avoiding arithmetic progressions of length $6$. We also prove a lower bound of $\frac{1}{2}$ on the lower density of subsets of positive integers that can be permuted to avoid arithmetic progressions of length $4$, sharpening the lower bound of $\frac{1}{3}$ from (LeSaulnier and Vijay, 2011). In addition, we generalize several results about forbidden arithmetic progressions to construct permutations avoiding generalized arithmetic progressions.
\end{abstract}

\section{Introduction}
Davis et al \cite{dga} proved that any permutation of the positive integers contains an arithmetic progression of length $3$, and they also constructed permutations of the positive integers avoiding arithmetic progressions of length $5$. As a result, they noted that there exist permutations of the integers that avoid arithmetic progressions of length $7$.

The results of Davis et al leave open the problem of whether there exist permutations of the positive integers avoiding arithmetic progressions of length $4$, and whether there exist permutations of the integers avoiding arithmetic progressions of length $4$, $5$, or $6$. These open problems are also mentioned in \cite{ega}.

In \cite{ega}, Erdos and Graham also asked whether the positive integers can be partitioned into two sets, both of which can be permuted to avoid arithmetic progressions of length $3$. This question is still unsolved, but a possible strategy to solve it was suggested in \cite{dga, ab}. 

Define $\alpha_{\mathbb{Z}^{+}}(k)$ to be the supremum of $lim sup_{n \rightarrow \infty} \frac{S \cap [1,n]}{n}$ over all sets $S$ of positive integers that can be permuted to avoid arithmetic progressions of length $k$, and similarly define $\beta_{\mathbb{Z}^{+}}(k)$ to be the supremum of $lim inf_{n \rightarrow \infty} \frac{S \cap [1,n]}{n}$ over all sets $S$ of positive integers that can be permuted to avoid arithmetic progressions of length $k$.

In \cite{ab}, LeSaulnier and Vijay noted that the answer to Erdos and Graham's question would be no if $\alpha_{\mathbb{Z}^{+}}(3)+ \beta_{\mathbb{Z}^{+}}(3) < 1$. They showed that $\alpha_{\mathbb{Z}^{+}}(3) \geq \frac{1}{2}$ and  $\beta_{\mathbb{Z}^{+}}(3) \geq \frac{1}{4}$, conjecturing that these lower bounds were tight and that there was no way to partition the positive integers in the way that Erdos and Graham described. They also proved that $\alpha_{\mathbb{Z}^{+}}(4) = 1$ and $\beta_{\mathbb{Z}^{+}}(4) \geq \frac{1}{3}$.

Davis et al \cite{dga} proved that there are permutations of the positive integers that avoid arithmetic progressions of length $5$, so $\alpha_{\mathbb{Z}^{+}}(k) = \beta_{\mathbb{Z}^{+}}(k) = 1$ for all $k \geq 5$. Thus the only open problem for $\alpha_{\mathbb{Z}^{+}}$ is evaluating $\alpha_{\mathbb{Z}^{+}}(3)$, and the only open problems for $\beta_{\mathbb{Z}^{+}}$ are evaluating $\beta_{\mathbb{Z}^{+}}(3)$ and $\beta_{\mathbb{Z}^{+}}(4)$.

In Section \ref{density}, we construct a permutation of the integers that avoids arithmetic progressions of length $6$. We also prove that $\beta_{\mathbb{Z}^{+}}(4) \geq \frac{1}{2}$, sharpening the bound from \cite{ab}. In addition, we prove density bounds for sets of integers rather than just positive integers. For these results, we define analogues of $\alpha_{\mathbb{Z}^{+}}(k)$ and $\beta_{\mathbb{Z}^{+}}(k)$ for the integers. 

The upper density function $\alpha_{\mathbb{Z}}(k)$ is the supremum of $lim sup_{n \rightarrow \infty} \frac{S \cap [-n,n]}{n}$ over all sets $S$ of integers which can be permuted to avoid arithmetic progressions of length $k$, and the lower density function $\beta_{\mathbb{Z}}(k)$ is the supremum of $lim inf_{n \rightarrow \infty} \frac{S \cap [-n,n]}{n}$ over all sets $S$ of positive integers which can be permuted to avoid arithmetic progressions of length $k$. 

In the other sections of the paper, we prove results about generalized arithmetic progressions, where we use the term $(r_{1}, \dots, r_{k-1})$ $k$-progression to refer to a sequence of $k$ numbers of the form $a, a+r_{1}d, \dots, a+\sum_{i = 1}^{k-1} r_{i}d$. We show that every permutation of the positive integers contains an $(r, s)$ $3$-progression in Section \ref{prog3}, and we find lower bounds on the number of permutations of $1, \dots, n$ that avoid certain generalized arithmetic progressions in Section \ref{finp}. Density bounds for permutations avoiding generalized arithmetic progressions are in Section \ref{gap}.

\section{Density bounds}\label{density}

First we construct a permutation of the integers avoiding arithmetic progressions of length $6$. In Section \ref{gap}, we show how a similar construction can be used to avoid certain generalized arithmetic progressions of length $6$.

\begin{prop}\label{mainres}
There exist permutations of the integers avoiding arithmetic progressions of length $6$.
\end{prop}

\begin{proof}
Define a sequence of intervals $A_{0}, A_{1}, \dots$ so that $A_{i} = [10^{i},10^{i+1})$ and $B_{0}, B_{1}, \dots$ so that $B_{i}$ contains the additive inverses of the elements in $A_{i}$. Let $X_{i}^{*}$ be obtained by rearranging $A_{i} \cup B_{i}$ so that it contains no arithmetic progression of length $3$. Finally, consider the permutation $0 X_{0}^{*} X_{1}^{*} \dots$.

Suppose for contradiction that the permutation contained an arithmetic progression of length $6$. Then the second element of the progression is in some block of the form $X_{i}^{*}$. If the third element of the progression is in the same block $X_{i}^{*}$ as the second element, then the fourth element must be in a different block. Since the absolute value of the difference between the first two elements in the progression is at most $2 \times 10^{i+1}$, we can conclude that the fourth, fifth, and sixth elements must all be in $X_{i+1}^{*}$, which is a contradiction. If the third element of the progression is in a different block from the second element, then the third, fourth, and fifth elements of the progression are forced to be in the same block, a contradiction.
\end{proof}

\begin{cor}
$\alpha_{\mathbb{Z}}(k) = \beta_{\mathbb{Z}}(k) = 1$ for all $k \geq 6$
\end{cor}

The bound $\beta_{\mathbb{Z}^{+}}(4) \geq \frac{1}{3}$ was proved in \cite{ab}. We sharpen this bound to $\frac{1}{2}$ below. 

\begin{prop}
 $\beta_{\mathbb{Z}^{+}}(4) \geq \frac{1}{2}$
\end{prop}

\begin{proof}
Define intervals $I_{1}, I_{2}, \dots$ of the form $I_{n} = [\ceil{a^{n}},\floor{ba^{n}}]$ for parameters $a, b$ with $1 < b < a$ to be chosen later, and let $I_{n}^{*}$ be obtained by rearranging $I_{n}$ so that it contains no arithmetic progression of length $3$. The set of positive integers consisting of the elements of the union of the $I_{n}$'s has lower density $\frac{b-1}{a-1}$ and upper density $(1-\frac{1}{b})(1+\frac{1}{a-1})$. 

For the permutation $I_{1}^{*} I_{2}^{*} \dots$ to avoid arithmetic progressions of length $4$, it suffices for $\frac{a}{b} - (\frac{a}{b}-1)2 \leq 0$ since any arithmetic progression of length $3$ would have to be contained in multiple $I_{n}^{*}$'s. This is equivalent to $\frac{a}{b} \geq 2$, so if we let $a = 2b$, then the lower density is $\frac{b-1}{a-1} = \frac{b-1}{2b-1}$. Since $b$ can be arbitrarily large, it follows that $\beta_{\mathbb{Z}^{+}}(4) \geq \frac{1}{2}$.
\end{proof}

In \cite{ab}, the authors proved that $\alpha_{\mathbb{Z}^{+}}(3) \geq \frac{1}{2}$ and $\beta_{\mathbb{Z}^{+}}(3) \geq \frac{1}{4}$, conjecturing that these bounds were tight. We find the same lower bound on upper density for permutations of the full set of integers avoiding arithmetic progressions of length $3$, but a lesser lower bound on lower density.

\begin{prop}
$\alpha_{\mathbb{Z}}(3) \geq \frac{1}{2}$ and $\beta_{\mathbb{Z}}(3) \geq \frac{1}{6}$
\end{prop}

\begin{proof}
Define a sequence of intervals $A_{1}, A_{2}, \dots$ so that $A_{i} = [5^{i}, \floor{\frac{5}{3} 5^{i}}]$ and $B_{1}, B_{2}, \dots$ so that $B_{i}$ contains the additive inverses of the elements in $A_{i}$. Let $X_{i}^{*}$ be obtained by rearranging $A_{i} \cup B_{i}$ so that it contains no arithmetic progression of length $3$. Finally, consider the sequence $X_{1}^{*}X_{2}^{*}\dots$.

First, note that the sequence has upper density $\frac{1}{2}$ and lower density $\frac{1}{6}$. Moreover, note that $\floor{\frac{5}{3}5^{i}}-(-\floor{\frac{5}{3}5^{i}})< 5^{i+1}-\floor{\frac{5}{3}5^{i}}$, so the second and third terms of a $3$-term arithmetic progression cannot be in different blocks. Also, $(5^{i+1}-\floor{\frac{5}{3}5^{i}})+5^{i+1} > \floor{\frac{5}{3}5^{i+1}}$, so the second and third terms of a $3$-term arithmetic progression cannot be in the same block. Thus there is no arithmetic progression of length $3$.
\end{proof}

\section{$3$-progression containment in permutations of the positive integers}\label{prog3}

Both results below use variations of the proof method introduced in \cite{dga} and also used in \cite{ab}.

\begin{prop}
For each integer $k > 1$, every permutation of the positive integers contains an arithmetic progression of length $3$ with difference not divisible by $k$.
\end{prop}

\begin{proof}
Without loss of generality, we may assume that $k$ is prime. The case $k = 2$ was proved in \cite{ab}. As in \cite{ab}, note that every permutation of $\left\{1,2,\dots,11\right\}$ with first element $2$ and second element $1$ contains an arithmetic progression of length $3$ with difference not divisible by $k$. This is trivially true for $k \geq 7$ and is checked with a computer for $k = 3$ and $k = 5$. Let $P = p_1 p_2 \dots$ be any permutation of the positive integers. Let $c$ be the least index such that $p_{c}-p_{1}$ is not divisible by $k$ and $p_{c} > p_{1}$, and let $p_{j} = max(p_{1},\dots,p_{c-1})$. 

If $p_{j} < 2p_{c}-p_{1}$, then $p_{1},p_{c},2p_{c}-p_{1}$ is an arithmetic progression of length $3$ in $P$. If $p_{j} \geq 2p_{c}-p_{1}$, then let $d = p_{j}-p_{c}$. Note that $d$ is not divisible by $k$, and $p_{j}$ occurs before $p_{j}-d$ in $P$, so we can apply the result in the first sentence of this proof to the set $\left\{p_{j}-d, p_{j}, p_{j}+d, \dots, p_{j}+9d\right\}$ to obtain an arithmetic progression of length $3$ in $P$ with difference not divisible by $k$.
\end{proof}

Although the result below implies that permutations of the positive integers always contain $(r, s)$ $3$-progressions, in the next section we will show how to find permutations of $1, \dots, n$ that avoid $(r, s)$ $3$-progressions for all $n > 0$.

\begin{prop}
For all positive integers $r$ and $s$, every permutation of the positive integers contains an $(r, s)$ $3$-progression.
\end{prop}

\begin{proof}
Let $P$ be an arbitrary permutation of the positive integers. Let $a_{0}, a_{1}, a_{2}, \dots$ denote the subsequence of $P$ consisting of elements that are greater than all elements to their left, so $a_{0}$ is the first element of $P$. By the pigeonhole principle, there are $2$ elements $a_{i} \equiv a_{j} \mod r$ with $i < j$. Then $a_{i}, a_{j}, a_{j}+\frac{s}{r}(a_{j}-a_{i})$ is an $(r, s, 3)$ progression in $P$.
\end{proof}

\section{Finite permutations avoiding progressions}\label{finp}

Let $\theta_{k}(n)$ denote the number of permutations of $1, \dots, n$ that avoid arithmetic progressions of length $k$. The first result below generalizes the well-known lower bound for the number of permutations of $1, \dots, n$ avoiding arithmetic progressions of length $3$. Later in the section, we are able to extend this same lower bound to certain $(r_{1}, \dots, r_{k-1})$ $k$-progressions.

\begin{prop}
For all $\epsilon > 0$, $\theta_{k}(n) \geq (k-1)!^{(\frac{1}{k-2}-\epsilon)n}$ if $n$ is a sufficiently large power of $k-1$
\end{prop}

\begin{proof}
We prove this by extending the lower bound proof used for $k = 3$. First, note that $\theta_{k}((k-1)n) \geq (k-1)! \theta_{k}(n)^{k-1}$ for all $k, n > 0$ since we can build permutations of $\left\{1, \dots, (k-1)n\right\}$ avoiding arithmetic progressions of length $k$ using permutations of $1, \dots, n$ avoiding arithmetic progressions of length $k$. Specifically let $A_{1}, \dots, A_{k-1}$ denote the subsets of $\left\{1, \dots, (k-1)n\right\}$ such that the elements in $A_{i}$ are congruent to $i$ mod $k-1$. For each $i$, there are $\theta_{k}(n)$ choices for how to permute the elements of $A_{i}$ and avoid arithmetic progressions of length $k$. We can concatenate the $A_{i}^{*}$'s in any order for a total of $(k-1)!$ possible orderings. 

First note that $\theta_{k}(k-1) = (k-1)!$ and in general $\theta_{k}((k-1)^n) = (k-1)!^{\sum_{i = 0}^{n-1}(k-1)^{i}}$. Observe that for all $\epsilon > 0$, there exists $N > 0$ such that for all $n > N$, $\frac{\sum_{i = 0}^{n-1}(k-1)^{i}}{(k-1)^n}-\frac{1}{k-2} < \epsilon$. 
\end{proof}

The result below is used in the lower bound constructions for the density proofs in the last section of this paper. 

\begin{prop}
For all positive integers $r$ and $s$ such that $2$ divides neither $r$ nor $s$, there exists a permutation of $1, \dots, n$ avoiding $(r, s)$ $3$-progressions.
\end{prop}

\begin{proof}
The permutation is obtained recursively. For $n \leq 2$, any permutation will work. For $n > 2$, split the integers $1, \dots, n$ into evens and odds, permute the evens and odds so that they avoid $(r, s)$ $3$-progressions, and then concatenate the two permutations with evens first. If there was an $(r, s)$ $3$-progression in the concatenation, then it could not be fully contained in the evens or fully contained in the odds, so without loss of generality suppose that the first two elements are in the evens and the last element is in the odds. Since $2$ divides neither $r$ nor $s$, this is a contradiction since $2$ divides the difference between the first two elements, but not the last two.
\end{proof}

Observe that the proof of the first result in this section can be applied to generalized arithmetic progressions. Specifically we obtain the following recursive inequality for generalized $(r_{1}, \dots, r_{k-1})$ $k$-progressions for which $k-1$ does not divide $r_{1}, \dots, r_{k-1}$.

\begin{prop}
Suppose that $k-1$ does not divide $r_{1}, \dots, r_{k-1}$. If $a$ is the number of permutations of $1, \dots, n$ that avoid $(r_{1}, \dots, r_{k-1})$ $k$-progressions and $b$ is the number of permutations of $1, \dots, n+1$ that avoid $(r_{1}, \dots, r_{k-1})$ $k$-progressions, then the number of permutations of $(k-1)n+j$ that avoid $(r_{1}, \dots, r_{k-1})$ $k$-progressions for $0 \leq j < k-1$ is at least $(k-1)! a^{k-1-j}b^{j}$.
\end{prop}

\begin{cor}
Suppose that $k-1$ does not divide $r_{1}, \dots, r_{k-1}$. For all $\epsilon > 0$, the number of permutations of $1, \dots, n$ avoiding $(r_{1}, \dots, r_{k-1})$ $k$-progressions is at least $(k-1)!^{(\frac{1}{k-2}-\epsilon)n}$ if $n$ is a sufficiently large power of $k-1$.
\end{cor}

\section{Density bounds for generalized arithmetic progressions}\label{gap}

Davis et al \cite{dga} asked for bounds on $\alpha_{\mathbb{Z}^{+}}(3)$ and $\beta_{\mathbb{Z}^{+}}(3)$, and the lower bounds  $\alpha_{\mathbb{Z}^{+}}(3) \geq \frac{1}{2}$ and $\beta_{\mathbb{Z}^{+}}(3) \geq \frac{1}{4}$ were proved in \cite{ab}. It is natural to generalize the problem from Davis et al to bound $\alpha_{\mathbb{Z}^{+}}(P)$ and $\beta_{\mathbb{Z}^{+}}(P)$ for any generalized arithmetic progression $P$. 

We extend the lower bounds from \cite{ab} to $(r, s)$ $3$-progressions in the result below, obtaining lower bounds on the upper and lower densities in terms of $r$ and $s$.

\begin{prop}
Suppose that $2$ divides neither $r$ nor $s$. Then there exist sets of positive integers with lower density $\frac{r s}{(r+s)^2}$ and upper density $\frac{s}{r+s}$ that avoid $(r, s)$ $3$-progressions.
\end{prop}

\begin{proof}
Define intervals $I_{1}, I_{2}, \dots$ of the form $I_{n} = [\ceil{a^{n}},\floor{ba^{n}}]$ for parameters $a, b$ with $1 < b < a$ to be chosen later, and let $I_{n}^{*}$ be obtained by rearranging $I_{n}$ so that it contains no $(r, s)$ $3$-progression. The set of positive integers consisting of the elements of the union of the $I_{n}$'s has lower density $\frac{b-1}{a-1}$ and upper density $(1-\frac{1}{b})(1+\frac{1}{a-1})$. 

For the permutation $I_{1}^{*} I_{2}^{*} \dots$ to avoid $(r, s)$ $3$-progressions, it suffices for $(b a^{i} - 0) \frac{s}{r} \leq a^{i+1}-b a^{i}$ and $(a^{i+1}-b a^{i}) \frac{s}{r}+a^{i+1} > b a^{i+1}$. This is equivalent to $b \leq \frac{a}{1+s/r}$ and $b < \frac{1+s/r}{1+\frac{s/r}{a}}$. Since $b$ can be arbitrarily close to $\frac{1+s/r}{1+\frac{s/r}{a}}$, it follows that we can obtain upper density $\frac{s}{r+s}$ and lower density $\frac{r s}{(r+s)^2}$ with $a$ approaching $\frac{r^2 + r s + s^2}{r^2}$ and $b =  \frac{r^2 + r s + s^2}{r^2+r s}$.
\end{proof}

Erdos and Graham \cite{ega} asked whether the positive integers can be partitioned into two sets which can both be permuted to avoid arithmetic progressions of length $3$. The conjecture in \cite{ab} that $\alpha_{\mathbb{Z}^{+}}(3) = \frac{1}{2}$ and $\beta_{\mathbb{Z}^{+}}(3) = \frac{1}{4}$ would negatively answer this question, since it would suffice to prove that  $\alpha_{\mathbb{Z}^{+}}(3)+\beta_{\mathbb{Z}^{+}}(3) < 1$. 

It is natural to ask Erdos and Graham's question for $(r, s)$ $3$-progressions in general, since the original question is just the case of $(1,1)$ $3$-progressions. We observe that $\frac{r s}{(r+s)^2}+\frac{s}{r+s} < 1$ for all $r, s > 0$, so if the bounds in our last proof are tight, then the positive integers could not be partitioned into two sets that can both be permuted to avoid $(r, s)$ $3$-progressions.

In the next result, we obtain the same lower bound for the upper density when we consider all integers instead of just positive integers, but we get a different lower bound on the lower density. 

\begin{prop}
Suppose that $2$ divides neither $r$ nor $s$. Then there exist sets of integers with lower density $\frac{r s}{(r+s)(r+2s)}$ and upper density $\frac{s}{r+s}$ that avoid $(r, s)$ $3$-progressions.
\end{prop}

\begin{proof}
Define intervals $I_{1}, I_{2}, \dots$ of the form $I_{n} = [\ceil{a^{n}},\floor{ba^{n}}]$ for parameters $a, b$ with $1 < b < a$ to be chosen later, and $J_{1}, J_{2}, \dots$ such that $J_{n}$ contains the additive inverses of the elements of $I_{n}$. Let $X_{n}^{*}$ be obtained by rearranging $I_{n} \cup J_{n}$ so that it contains no $(r, s)$ $3$-progression. 

For the permutation $X_{1}^{*} X_{2}^{*} \dots$ to avoid $(r, s)$ $3$-progressions, it suffices for $(b a^{i} - (-b a^{i})) \frac{s}{r} < a^{i+1}-b a^{i}$ and $(a^{i+1}-b a^{i}) \frac{s}{r}+a^{i+1} > b a^{i+1}$. This is equivalent to $b < \frac{a}{1+2s/r}$ and $b < \frac{1+s/r}{1+\frac{s/r}{a}}$. It follows that we can obtain upper density $\frac{s}{r+s}$ and lower density $\frac{r s}{(r+s)(r+2s)}$ with $a$ approaching $\frac{r^2 + 2r s + 2s^2}{r^2}$ and $b = \frac{r^2 + 2r s + 2s^2}{r^2+2r s}$.
\end{proof}

As with the last result, we see that $\frac{r s}{(r+s)(r+2s)}+\frac{s}{r+s} < 1$ for all $r, s > 0$, so if the bounds in our last proof are tight, then the integers could not be partitioned into two sets that can both be permuted to avoid $(r, s)$ $3$-progressions.

The last proof below is a generalization of the construction of the permutation of the integers avoiding arithmetic progressions of length $6$ in Section \ref{density} to a certain class of generalized arithmetic progressions of length $6$.

\begin{prop}
For all $r, s > 0$ for which $2$ divides neither $r$ nor $s$, there exist permutations of the integers that avoid $(r^{4},r^{3}s,r^{2}s^{2},r s^{3},s^{4})$ $6$-progressions.
\end{prop}

\begin{proof}
Define a sequence of intervals $A_{0}, A_{1}, \dots$ so that $A_{i} = [a^{i},a^{i+1})$ and $B_{0}, B_{1}, \dots$ so that $B_{i}$ contains the additive inverses of the elements in $A_{i}$. Choose $a > 1+2(\frac{s}{r}+\frac{s^2}{r^2}+\frac{s^3}{r^3})$ to be an integer. Let $X_{i}^{*}$ be obtained by rearranging $A_{i} \cup B_{i}$ so that it contains no $(r,s)$ $3$-progression. Finally, consider the permutation $P = 0 X_{0}^{*} X_{1}^{*} \dots$.

Suppose that $P$ contained an $(r^{4},r^{3}s,r^{2}s^{2},r s^{3},s^{4})$ $6$-progression. The rest of the proof is the same as Proposition \ref{mainres}.
\end{proof}

Bounding the upper and lower densities of other generalized arithmetic progressions is a possible direction for future research. Also the main open problems from \cite{dga} are still unsolved, including (1) the existence of a permutation of the positive integers avoiding arithmetic progressions of length $4$ and (2) a partition of the positive integers into two subsets, each which can be permuted to avoid arithmetic progressions of length $3$. These problems could also be investigated for generalized arithmetic progressions.

\end{document}